\documentclass[a4paper,11pt]{article}
\usepackage{a4wide}
\usepackage[british]{babel}
\usepackage{extarrows}
\usepackage{enumerate}
\usepackage{graphics}
\usepackage{calc}

\usepackage{amsmath,amsfonts,amssymb}
\usepackage[thmmarks,hyperref,amsmath]{ntheorem}
\theoremseparator{.}
\newtheorem{Prop}{Proposition}[section]
\newtheorem{Thm}[Prop]{Theorem}
\newtheorem{Lem}[Prop]{Lemma}

\newtheorem{Def}[Prop]{Definition}

\newtheorem{theorem}{\bf Theorem}[section]

\newtheorem{corollary}[theorem]{\bf Corollary}

\newenvironment{proof}{\noindent{\em Proof:}}{\quad \hfill$\Box$\vspace{2ex}}

\theoremstyle{nonumberplain}
\theoremheaderfont{\scshape}
\theorembodyfont{\upshape}
\theoremsymbol{\raisebox{1pt}{\scriptsize\ensuremath{\Box}}}

\def \bI {\mathbb I}
\def \bN {\mathbb N}

\def \bR {\mathbb R}
\def \K  {\bf K}
\def \G  {\bf G}
\def \bO {\mathbb O}

\def \bC {\mathbb C}
\def \bA {\mathbb A}

\def \by {{\bf y}}

\def \bb {{\bf b}}
\def \bc {{\bf c}}
\def \bd {{\bf d}}
\def \be {{\bf e}}

\def \bs {{\bf s}}
\def \bt {{\bf t}}

\def \cB {{\cal B}}

\def \cF {{\cal F}}

\def \cI {{\cal I}}

\def \cL {{\cal L}}

\def \cN {{\cal N}}
\def \cS {{\cal S}}

\def \cE {{\cal E}}

\def \supp {\,{\rm supp}\,}

\newcommand{\0}{{\bf 0}}
\newcommand{\BK}{{\mathcal B}_{\bf K}}
\newcommand{\CC}{{\rm C}}

\newcommand{\x}{{\bf x}}
\newcommand{\y}{{\bf y}}
\newcommand{\z}{{\bf z}}

\def\sm{\setminus}

\begin{document}

\newlength{\bibitemsep}\setlength{\bibitemsep}{.2\baselineskip plus .05\baselineskip minus .05\baselineskip}
\newlength{\bibparskip}\setlength{\bibparskip}{0pt}
\let\oldthebibliography\thebibliography
\renewcommand\thebibliography[1]{%
  \oldthebibliography{#1}%
  \setlength{\parskip}{\bibitemsep}%
  \setlength{\itemsep}{\bibparskip}%
}

\vspace{2mm} \baselineskip 15pt
\renewcommand{\baselinestretch}{1.10}
\parindent=16pt  \parskip=2mm
\rm\normalsize\rm

\title{\bf Vector-valued Reproducing Kernel Banach Spaces with Group Lasso Norms\thanks{Supported by National Natural Science Foundation of China under grant 12371103, and by Guangdong Basic and Applied Basic Research Foundation (2024A1515011194).}}
\author{Liangzhi Chen\thanks{School of Computer Science and Technology, Dongguan University of Technology, Dongguan 523808, P. R. China. E-mail address: {\it chenlzh23@mail3.sysu.edu.cn.}},
\quad Haizhang Zhang\thanks{The corresponding author. School of Mathematics (Zhuhai), Sun Yat-sen University, Zhuhai, P.R. China. E-mail address: zhhaizh2@sysu.edu.cn.} , \quad and\quad Jun Zhang\thanks{Department of Psychology and Department of Mathematics, University of Michigan, Ann Arbor, MI 48109, USA. E-mail address: {\it junz@umich.edu}.}}
\date{}\maketitle

\begin{abstract} Focusing on establishing a mathematical basis for kernel methods in sparse multi-task learning, we explore the theory of vector-valued reproducing kernel Banach spaces (RKBSs) endowed with $\ell_{p,1}$-norms ($1\le p\le +\infty$), encompassing both the sparse learning case when $p=1$ and the group lasso when $p=2$. We develop RKBSs equipped with these group lasso norms that support the linear representer theorem for regularized learning frameworks. Additionally, we introduce reproducing kernels admissible for this construction. Such reproducing kernels are applicable to sparse multi-task learning with group lasso norms.

\noindent{\bf Keywords:}  vector-valued spaces, reproducing kernel Banach spaces, sparse multi-task learning, the representer theorem, admissible kernels
\end{abstract}

\section{Introduction}

Learning theory is concerned with the identification of effective predictors based on limited datasets. However, addressing such issues frequently leads to ill-posed problems, as noted in \cite{Morozov,Tikhonov}. Regularization is a commonly employed technique to mitigate these challenges. The process of learning within function spaces can be articulated as an optimization problem that encompasses both an error term and a regularization component. This can be expressed mathematically as follows:
\begin{equation}\label{regularization scheme}
	\min_{f\in\cF} \ L(f(\x),\y)+\lambda \Omega(f),
\end{equation}
where \(\cF\) denotes a function space defined over a dataset \(X\), \((\x,\y)\) represents a collection of input and output data, \(\lambda > 0\) is a regularization parameter, \(L\) signifies the error function, and \(\Omega\) is referred to as the regularizer function.

Classical instances of the regularization scheme (\ref{regularization scheme}) involve the use of norms in Hilbert spaces, a topic that has been thoroughly examined in the literature (see \cite{Berlinet,Cuckersmale,Scholkopf}). Notably, the study of learning within reproducing kernel Hilbert spaces (RKHSs) has received significant attention in the fields of machine learning \cite{Aronszajn,Scholkopf,Shawe}, statistical learning \cite{Berlinet,Vapnik}, and stochastic processes \cite{Rasmussen}, among others, over the past several decades. Several factors contribute to the efficacy of learning methodologies in RKHSs. Firstly, kernels facilitate the measurement of similarity between input data points through the application of "kernel tricks." Secondly, an RKHS is defined as a Hilbert space of functions on a set \(X\) where point evaluations are continuous linear functionals. The sample data utilized for learning are typically modeled as point evaluations of the unknown target function. Lastly, according to the Riesz representation theorem, the point evaluation functionals on \(X\) can be expressed in terms of their corresponding reproducing kernel. These elements culminate in the well-known representer theorem \cite{Argyriou,Kimeldorf}, which is particularly advantageous for learning approaches in high-dimensional or infinite-dimensional spaces.

 In contrast to Hilbert spaces, Banach spaces exhibit a greater variety of geometric structures and norms that can be useful for learning applications. Recent theoretical investigations have focused on learning within scalar-valued reproducing kernel Banach spaces (RKBSs) \cite{Micchelli1,L.Shi,Songl^1,Songl^1II,Tong,Xiao,ZhangXu,ZhangZhang} as well as in the context of \textit{multi-task learning} \cite{Alvarez,Burbea,Caponnetto,Carmeli,LinZhang,Micchelli2,ZhangXuZhangQ,ZhangZhangJ}. In particular, research concerning $\ell_1$-norm RKBSs \cite{Songl^1} has garnered significant interest. This is primarily attributed to the fact that $\ell_1$-norm regularization \cite{Tibshirani} in single-task learning scenarios frequently yields sparse solutions \cite{Candes,Donoho,Tropp}, which are highly sought after in the field of machine learning. The concept of sparsity is crucial for the extraction of relatively low-dimensional features from sample data, which typically reside in high-dimensional spaces. In summary, we make a brief list of former researches on RKBSs:
 \begin{enumerate}
 	\item Scalar-valued RKBSs \cite{ZhangXu,ZhangZhang} and vector-valued RKBSs \cite{ZhangZhangJ} built on uniformly convex and uniformly smooth Banach spaces via semi-inner products \cite{Lumer}.
 	\item Scalar-valued RKBS with the $\ell_1$-norm \cite{Songl^1,Songl^1II}.
 	\item The $s$-norm scalar-valued RKBSs \cite{Xu} developed via dual-bilinear forms and the generalized Mercer kernels.
 	\item Vector-valued RKBSs with the $\ell_1$-norm \cite{LinZhang}.
 	\item Generic definitions and unified framework of construction of scalar-valued RKBSs \cite{LinZhang2}.
 \end{enumerate}

Multi-task learning is increasingly prevalent in various applications. Traditional methods that rely on single-task learning techniques often make the unrealistic assumption that tasks are independent of one another, which typically results in suboptimal performance, particularly when dealing with small datasets. In contrast, multi-task learning leverages the interrelatedness of tasks to enhance the overall efficacy of the learning process. Many multi-task learning methodologies have been introduced to improve the performance of lasso in addressing these challenges, including the smoothly clipped absolute deviation \cite{Fan,Wang}, the adaptive lasso \cite{Zou}, the relaxed lasso \cite{Meinshausen}, the group lasso \cite{Yuan}, and the sparse group lasso \cite{Friedman,Simon}. Empirical studies \cite{Caruana,Evgeniou2005,Friedman,Meinshausen,Micchelli2} indicate that multi-task learning generally yields superior outcomes compared to single-task learning.

The primary aim of this paper is to establish a learning theory for vector-valued reproducing kernel Banach spaces (RKBSs) utilizing the \(\ell_{p,1}\) norms. Specifically, when \(p=1\), this framework simplifies to the \(\ell^1\)-norm vector-valued RKBS, which has been recently examined in \cite{LinZhang}. Our methodology is broader in scope and encompasses the significant group lasso scenario when \(p=2\). The initial objective is to construct an \(\ell_{p,1}\)-norm vector-valued RKBS based on admissible kernels, followed by the derivation of the representer theorem pertinent to regularized learning frameworks. These topics are addressed in Sections 3 and 4. The subsequent objective centers on the exploration of admissible kernels. In Section 5, we introduce a novel family of admissible kernels and engage in a discussion regarding kernel functions whose Lebesgue constants are bounded above by 1. Finally, a relaxed linear representer theorem is discussed in Section 6 to accommodate more reproducing kernels.

\section{Preliminaries and Notations}
Throughout this paper, $p$ always denotes a number in $[1,+\infty]$,
and $q$ is its conjugate number such that $1/p+1/q=1$. Let $\bN$ denote the set of all positive integers and set $\bN_k:=\{1,2,\dots,k\}$ for every $k\in\bN$.
Denote by $\bC,\bR$, and $\bR_+$ the set of complex numbers, real numbers, and nonnegative real numbers, respectively.

For a Banach space $\cB$, denote its dual Banach space by $\cB^*$. The zero element in $\cB$ is denoted by $\0$. Recall the classical Banach space $\ell^p$ consisting of elements $v=(v_1,x_2,\dots,v_k,\dots)$ with finite $\ell^p$-norm, which is defined for $1\le p<+\infty$ as
$$
\|v\|_p=\left(\sum_{i=1}^{\infty} |v_i|^p\right)^{1/p}
$$
and as $\|v\|_\infty=\sup\limits_{i}|v_i|$ for $p=+\infty$.
Let $(E,\|\cdot\|_p)$ be a finite-dimensional complex Euclidean space $E$ with the $\ell^p$-norm.
We then introduce an important notation of this paper
$$
\cB_p=\begin{cases}
              \ell^2,  & p=2;\\
              (\bC^n,\|\cdot\|_p), & p\neq 2.
      \end{cases}
$$
It is easy to see that $\cB_p^*=\cB_q$ and $\cB_p=\cB_p^{**}$.
Denote the bilinear form on $\cB_p\times\cB_q$ by $\langle \cdot,\cdot\rangle_p$.
Thus, for $x\in\cB_p,y\in\cB_q$,
$\langle x,y\rangle_p:=y(x)=x(y)=:\langle y,x\rangle_q$ and $|\langle x,y\rangle_p|\le \|x\|_p\|y\|_q$.

For two Banach spaces $\cE_1,\cE_2$, $\cL(\cE_1,\cE_2)$ denotes the space of all bounded linear operators from $\cE_1$ to $\cE_2$. It is also a Banach space with the operator norm
\begin{equation*}\label{normLE}
  \|A\|_{\cL(\cE_1,\cE_2)}:=\sup_{\alpha\in\cE_1 \atop \alpha\neq \0} \frac{\|A\alpha\|_{\cE_2}}{\|\alpha\|_{\cE_1}}.
\end{equation*}

For any nonempty set $\Omega$, denote by $\cB_p^\Omega$ the set of vector-valued functions on $\Omega$ that is nonzero on at most countable elements in $\Omega$. In other words, for every $C=(c_t)_{t\in\Omega}\in \cB_p^\Omega$, each $c_t$ is a vector in a fixed finite-dimensional Euclidean space, and $c_t\ne \0$ for at most countable elements $t\in\Omega$. We then introduce the following space
\begin{equation*}
  \ell_{p,1}(\Omega):=\left\{\CC=(c_t)_{t\in\Omega}\in \cB_p^\Omega:\|\CC\|_{p,1}=\sum_{t\in\Omega}\|c_t\|_p<+\infty\right\}.
\end{equation*}

We denote the set of $m$ samplings in an input space $X$ by $\x=\{x_i\in X:i\in\bN_m\}$, and the corresponding observations by
$\y=\{y_i\in \cB_q:i\in\bN_m\}$.
For later convenience, we introduce the following notation. For a function $\K:X\times X\to \cL(\cB_p,\cB_q)$, denote by
$$
 {\K}[\x]:=[{\K}(x_i,x_j):i,j\in\bN_m]
$$
an $m\times m$ matrix with entries in $\cL(\cB_p,\cB_q)$.
It has two associated vectors denoted by
$$
 {\K}^\x(x):=({\K}(x_i,x):i\in\bN_m),\ \ x\in X
$$
and
$$
{\K}_\x(x):=({{\K}}(x,x_i):i\in\bN_m)^T,\ \ x\in X.
$$

\subsection{Reproducing kernel Banach spaces of vector-valued functions}
Before giving a formal definition of RKBSs of vector-valued functions, we recall some terminologies.

\begin{Def}\label{BSofVVF}{\rm\cite{ZhangXu}}
A space $\cB$ of certain vector-valued functions on a prescribed set $X$ is called a \textbf{Banach space of vector-valued functions} if the point
 evaluation functionals are consistent with the
norm on $\cB$ in the sense that for all $f\in \cB$, $\|f\|_\cB=0$ if and only if $f(x)=\0$ for every $x\in X$.
A Banach space $\cB$ of vector-valued functions on $X$ is said to be a \textbf{pre-RKBS} on
  $X$ if point evaluations are continuous linear functionals on $\cB$.
\end{Def}

To accommodate the main purpose of this paper, we present a slightly different version of RKBSs of vector-valued functions from \cite{ZhangZhangJ}.  Denote a space $\cB$ with the norm $\|\cdot\|_\cB$ by $(\cB,\|\cdot\|_\cB)$.

\begin{Def}\label{defiRKBS}
   The Banach spaces $(\cB,\|\cdot\|_\cB)$ and $(\cB^\#,\|\cdot\|_{\cB^\#})$
   are called RKBSs of vector-valued functions from $X$ to $\cB_q$, provided that
  \begin{enumerate}
     \renewcommand{\labelenumi}{(\roman{enumi})}
     \item $\cB$ and $\cB^\#$ are pre-RKBS of vector-valued functions;
     \item There exists a kernel function ${\K}:X\times X\to \cL(\cB_p,\cB_q)$ 
           such that
              $$
                 {\K}(x,\cdot)c\in \cB,~{\K}(\cdot,x)c\in \cB^\#~\text{for~all~}x\in X,c\in \cB_p;
              $$
     \item There is a bilinear form $(\cdot,\cdot)_{\K}$ on $\cB\times\cB^\#$ for which the following reproducing properties
              $$
                 \left(f,{\K}(\cdot,x)c\right)_{{\K}}
                 =\langle f(x), c\rangle_{q},~({\K}(x,\cdot)c,g)_{{\K}}=\langle c,g(x)\rangle_{p}
              $$
           hold true for all
           $x\in X,c\in\cB_p,f\in \cB,~g\in \cB^\#$.
  \end{enumerate}
Under these assumptions, ${\K}$ is called the \textbf{reproducing kernel} of $\cB$ and $\cB^\#$.
\end{Def}

\subsection{Admissible kernels}
The requirements of a kernel function that can be used to construct a vector-valued RKBS with the $\ell_{p,1}$-norm
are formulated as follows.

\begin{Def}[\bf Admissible Kernels]\label{defiadmissible}
A kernel ${\K}:X\times X\to \cL(\cB_p,\cB_q)$ is \textbf{admissible} for the
construction of RKBS of vector-valued functions from $X$ to $\cB_q$ endowed with the $\ell_{p,1}$-norm if the following assumptions are satisfied.
\begin{enumerate}
    \renewcommand{\labelenumi}{(\bf A\arabic{enumi})}
  \item\label{A1} For any $m$ pairwise distinct sampling points $\x\subseteq X$, the matrix
       $$
          {\K}[\x]:=\left[{\K}(x_i,x_j):k,j\in\bN_m\right]\in \cL(\cB_p,\cB_q)^{m\times m}
       $$
      is invertible in the sense that there exists a ${\K}'[\x]\in \cL(\cB_q,\cB_p)^{m\times m}$, such that
       $$
          {\K}[\x]{\K}'[\x]={\rm diag}(\bI_q,\dots,\bI_q)_m
          $$
          and
          $${\K}'[\x]{\K}[\x]={\rm diag}(\bI_p,\dots,\bI_p)_m
       $$
      where $\bI_p\in\cL(\cB_p,\cB_p)$ is the identity operator on $\cB_p$, and ${\rm diag}(\bI_p,\dots,\bI_p)_m$ is
      an $m\times m$ matrix with diagonal entries $\bI_p$ and zero operator $\bO$ elsewhere.
      We simply denote ${\K}'[\x]$ by ${\K}[\x]^{-1}$ if no confusion is caused.
  \item\label{A2} The kernel ${\K}$ is bounded.
   That is, there exists $\kappa>0$ such that the operator norm
   $$
   \|{\K}(x,x')\|_{\cL(\cB_p,\cB_q)}\le \kappa
   $$
    for all $x,x'\in X$.
  \item\label{A3} For any pairwise distinct points $x_i\in X,i\in\bN$ and $(c_i)_{i\in\bN}\in \ell_{p,1}(\bN)$, if
       $\sum\limits_{i\in \bN}{\K}(x_i,x)c_i=\0$ for all $x\in X$, then $c_i=\0$ for all $i\in\bN$.
  \item\label{A4} For any pairwise distinct $x_1,x_2,\dots,x_m,x_{m+1}\in X$,
       $$
          \|{\K}[\x]^{-1}{\K}_\x(x_{m+1})\|_{p,1}:=\sup_{c\in\cB_p\atop c\neq \0}\frac{\|{\K}[\x]^{-1}{\K}_\x(x_{m+1})c\|_{p,1}}{\|c\|_p}
       $$
       is bounded above by $1$, where ${\K}[\x]^{-1}{\K}_\x(x_{m+1})$ is a linear operator from $\cB_p$ to $\cB_p^m$.
\end{enumerate}
\end{Def}

We denote the corresponding assumptions for the scalar case in \cite{Songl^1}
by ({\bf A\ref{A1}$'$})--({\bf A\ref{A4}$'$}).

We make some remarks on assumptions ({\bf A\ref{A1}}) and ({\bf A\ref{A4}}) in {Definition \ref{defiadmissible}}.
For ({\bf A\ref{A1}}), note that for $p<q$, we have $\ell^p\subseteq \ell^q$ and there do exist two linear operators
$$
A:\ell^p\to \ell^q, ~B:\ell^q\to\ell^p
$$
 such that $BA=\bI_p$ and $\|AB\|_{\cL(\ell_q,\ell_q)}=1$.
If both the linear operators $A,B$ are bounded, then most of the theoretic work in this paper
would hold for $\cB_p=\ell^p$. But unfortunately, for $1 \leq s \neq t \leq +\infty$, there do not exist two bounded linear operators $A: \ell_s \to \ell_t$ and $B: \ell_t \to \ell_s$ such that either $AB = \bI_t$ or $BA = \bI_s$. This result follows from the well-known Pitt's theorem (see, for example, \cite{Fabian}, p. 175), which states that if $1\le s < t<+\infty$, then any operator $A \in \mathcal{L}(\ell_t, \ell_s)$ must be compact. Consequently, if $A \in \mathcal{L}(\ell_s, \ell_t)$ and $B \in \mathcal{L}(\ell_t, \ell_s)$ with $s \neq t$, both compositions $AB$ and $BA$ are compact operators, and therefore cannot be the identity operator. This is the main reason why we have to assume $\cB_p~(p\neq 2)$ to be a finite-dimensional Euclidean space. Regarding assumption ({\bf A\ref{A4}}), it may initially appear counterintuitive, but it is crucial for admissible kernels. Its presence enables a linear representer theorem for learning in RKBSs, which we will explore further in this paper.

\subsection{Further preliminaries on matrix theory}
We discuss some useful facts about the operator norm $\|\cdot\|_{p,1}$ defined in ({\bf A\ref{A4}}).
 For an $m\times n$ operator matrix $B\in \cL(\cB_p,\cB_p)^{m\times n}$ and a vector $\bc=(c_1,c_2,\dots,c_n)^T\in\cB_p^n$,
  we have the following compatible inequality for $\|\cdot\|_{p,1}$,
\begin{equation}\label{Matrixnormcompatible}
   \begin{split}
      \|B{\bc}\|_{p,1}
      &=\|\sum_{k=1}^n B_k c_k\|_{p,1}\le\sum_{k=1}^n \|B_k c_k\|_{p,1} \\
      &\le\sum_{k=1}^n \|B_k\|_{p,1} \|c_k\|_{p}\le\max_{k\in\bN_n}\left(\|B_{k}\|_{p,1}\right)~\|{\bc}\|_{p,1},
   \end{split}
\end{equation}
where $B_{k}$ denotes the $k$-th column of $B$. When the entries of $\bc$ are scalar-valued, $\|\bc\|_{p,1}=\|\bc\|_1$.

Additionally, the inversion of a $2\times 2$ block matrix shown below will be frequently utilized throughout this paper:
\begin{equation}\label{matrixinversion}
  \left[
     \begin{array}{cc}
        A & B \\
        C & D \\
     \end{array}
   \right]^{-1}=
 \left[
     \begin{array}{ll}
       A^{-1}+A^{-1}BMCA^{-1} & -A^{-1}BM\\
       -MCA^{-1} & M\\
     \end{array}
   \right]
\end{equation}
where $M=(D-CA^{-1}B)^{-1}$.

\section{Construction of RKBSs}
To begin with, we will use a similar method as in \cite{Songl^1} to construct
a vector-valued Banach space with the $\ell_{p,1}$ based on a kernel
satisfying ({\bf A\ref{A2}}) and ({\bf A\ref{A3}}) in {Definition \ref{defiadmissible}}.

Let $X$ be a given input space. We shall construct the following two RKBSs of vector-valued functions from $X$ to $\cB_q$. The first one is
\begin{equation}\label{BK}
  \BK:=\biggl\{\sum_{x\in\supp \CC}{\K}(x,\cdot)c_x:\CC=(c_x)_{x\in X}\in \ell_{p,1}(X)\biggr\}
\end{equation}
with the norm
\begin{equation}\label{normBK}
  \biggl\|\sum_{x\in\supp \CC}{\K}(x,\cdot)c_x\biggr\|_{\BK}:=\sum_{x\in\supp \CC}\|c_x\|_p.
\end{equation}
And the second one is
 \begin{equation}\label{BK*}
   {\BK^\#}:=\left\{\sum_{x\in\supp \CC} {\K}(\cdot,x)c_x:\CC=(c_x)_{x\in X}\in \ell_{p,1}(X)\right\}
 \end{equation}
with the norm
\begin{equation}\label{normBK*}
 { \biggl\|\sum_{x\in\supp \CC} {\K}(\cdot,x)c_x\biggr\|_{\BK^\#}:=\sup_{y\in X}\biggl\|\sum_{x\in\supp \CC} {\K}(y,x)c_x\biggr\|_q.}
\end{equation}

\subsection{The bilinear form and point evaluations}
Denote
$$
\BK^0:=\left\{\sum_{i=1}^m{\K}(x_i,\cdot)c_i:x_i\in X,c_i\in\cB_p,i\in\bN_m~\text{for~all~}m\in\bN\right\}
$$
with the norm
\begin{equation}\label{normBKnotcomplete}
  \left\|\sum_{i=1}^n{\K}(x_i,\cdot)c_i\right\|_{\BK^0}:=\sum_{i=1}^n \left\|c_i\right\|_p,
\end{equation}
and a linear space
\begin{equation*}
   {\BK^{0,\#}}:=\left\{\sum_{i=1}^m{\K}(\cdot,x_i)c_i:x_i\in X,c_i\in\cB_p,i\in\bN_m,\ m\in\bN\right\}.
 \end{equation*}
 The above two linear spaces both consist of functions from $X$ to $\cB_q$.

We then define a bilinear form $(\cdot,\cdot)_{{\K}}$ on $\BK^0\times \BK^{0,\#}$ by
\begin{equation}\label{bilinearfinite}
  \left(\sum_{i=1}^m {\K}(x_i,\cdot)a_i,\sum_{j=1}^{m'} {\K}(\cdot,x'_j)b_j\right)_{{\K}}
  =\sum_{i=1}^m\sum_{j=1}^{m'}\langle a_i,{\K}(x_i,x'_j)b_j\rangle_{p},
\end{equation}
where $m,m'\in\bN$ and $x_i,x'_j\in X,~a_i,b_j\in\cB_p$ for $i\in\bN_m,j\in\bN_{m'}$.

By ({\bf A\ref{A3}}), we know that the norm in (\ref{normBKnotcomplete}) and the above bilinear form in (\ref{bilinearfinite})
are well-defined on their underlying spaces.

To proceed, we have to show that the point evaluation operators $\delta_x:\BK^{0}\to\cB_q,~x\in X$
or $\delta_x:\BK^{0,\#}\to\cB_q,~x\in X$ defined as follows
$$
 \delta_x(f)=f(x),~\text{where}~f\in\BK^0~\text{or}~\BK^{0,\#},
$$
 are continuous operators.

\begin{Prop}\label{Proppointevaluation}
The point evaluation operators are continuous on $\BK^0$ because
$$
\|\delta_x(f)\|_q\le \kappa\|f\|_{\BK^0}, ~~\text{for}~ f\in \BK^0,
$$
 where $\kappa>0$ is the constant in ({\bf A\ref{A2}}).
\end{Prop}
\begin{proof} Let $f=\sum_{i=1}^n {\K}(z_i,\cdot)a_i\in\BK^0$ with $z_i\in X,~a_i\in\cB_p,i\in\bN_n$.
Then we have
\begin{equation*}
\begin{split}
  \|f(x)\|_q
  &=\sup_{\|c\|_p\le 1} |\langle  f(x),c\rangle_q|\\
  &=\sup_{\|c\|_p\le 1} \left|\left(\sum_{i=1}^n {\K}(z_i,\cdot)a_i,{\K}(\cdot,x)c\right)_{\K}\right|\\
  &=\sup_{\|c\|_p\le 1} \left|\sum_{i=1}^n \langle a_i,{\K}(z_i,x)c\rangle_{p}\right|\\
  &\le \sup_{\|c\|_p\le 1} \sum_{i=1}^n \|a_i\|_p \sup_{i\in\bN_n}\|{\K}(z_i,x)c\|_q\\
  &\le \|f\|_{\BK} \sup_{\|c\|_p\le 1} \sup_{i\in\bN_n}\|{\K}(z_i,x)\|_{\cL(\cB_p,\cB_q)}\|c\|_p\\
  &\le \kappa \|f\|_{\BK}.
  \end{split}
\end{equation*}
This shows that the point evaluation operators are continuous on $\BK^0$.
\end{proof}

By \cite[Proposition 2.4]{Songl^1}, we know that the norm defined by
\begin{equation}\label{normBK*notcomplete}
  \left\| g \right\|_{\BK^{0\#}}:=\sup_{f\in\BK^0\atop f\neq 0}\frac{|(f,g)_{\K}|}{\|f\|_{\BK^0}}.
\end{equation}
 is well defined.
Moreover, using reasoning similar to that in Proposition \ref{Proppointevaluation},
we can show that the point evaluation operators on $\BK^{0,\#}$ are continuous and
 \begin{equation}\label{g(x)lessthan gBK}
 \begin{split}
   \|g(x)\|_q&=\sup_{\|c\|_p\le 1} |\langle g(x), c\rangle_q|\\
   &=\sup_{\|c\|_p\le 1} |({\K}(x,\cdot)c,g)_{{\K}}|\le \|g\|_{\BK^{0,\#}}
   \end{split}
 \end{equation}
for every $g\in\BK^{0,\#}$.

The norm defined as in (\ref{normBK*notcomplete}) has another equivalent but simpler form.
\begin{Prop}\label{PropnormBK*notcomplete}
For any $g\in\BK^{0,\#}$, it holds that
\begin{equation*}
  \|g\|_{\BK^{0,\#}}=\sup\limits_{x\in X}\|g(x)\|_q.
\end{equation*}
\end{Prop}
\begin{proof}
By (\ref{g(x)lessthan gBK}), we have $\sup\limits_{x\in X}\|g(x)\|_q\le \|g\|_{\BK^{0,\#}}$.
We shall prove the opposite direction.
For any $f\in\BK^0$, there exist pairwise distinct points
$x_i\in X,~c_i\in\cB_p,~i\in\bN_n$ such that
$$
f(x)=\sum_{i\in\bN_n} {\K}(x_i,x)c_i,~~~x\in X.
$$
Then, we have for every $g\in\BK^{0,\#}$,
\begin{equation*}
   \begin{split}
   |(f,g)_{{\K}}|
   &= \left|\left(\sum_{i\in\bN_n} {\K}(x_i,\cdot)c_i,g\right)_{{\K}}\right|\\
   &\le \sum_{i\in\bN_n} \left|\langle  c_i, g(x_i)\rangle_{p}\right|\\
   &\le \sum_{i\in\bN_n} \|c_i\|_p\sup_{i\in\bN_n}\|g(x_i)\|_q\\
   &\le \|f\|_{\BK} \sup_{x\in X}\|g(x)\|_q.
      \end{split}
\end{equation*}
It follows that $\|g\|_{\BK^\#}\le \sup\limits_{x\in X}\|g(x)\|_q$, which completes the proof.
\end{proof}

So far, we have introduced two normed vector spaces, $(\BK^0,\|\cdot\|_{\BK^0})$ and $(\BK^{0,\#},\|\cdot\|_{\BK^{0,\#}})$, each having continuous point evaluation functionals. Additionally, there is a bilinear form (\ref{bilinearfinite}) defined on the product space $\BK^0 \times \BK^{0,\#}$.

\subsection{Completion of $\BK^0$ and $\BK^{0,\#}$}
With the previous preparations, we are now ready to complete $\BK^0$ and $\BK^{0,\#}$ into RKBSs.
Just like the classical completion process, we simply add elements into $\BK^0$ and $\BK^{0,\#}$ to make them Banach spaces of functions.
For convenience, we use the notation $\cN_0$ to represent $\BK^0$ or $\BK^{0,\#}$.
Let $\{f_n:n\in \bN\}$ be a Cauchy sequence in $\cN_0$.
Then by {Proposition \ref{Proppointevaluation}} and the fact that $\cB_q$ is a Banach space,
for any $x\in X$, the sequence $\{f_n(x):n\in\bN\}$ is convergent to some point in $\cB_q$.
We denote this limit by $f(x)$, which defines a vector-valued function $f:X\to\cB_q$.
It is easy to see that $f$ is well-defined.
We then let $\cN$ be the set consisting of all such limit vector-valued functions with
the norm $\|f\|_{\cN}=\lim\limits_{n\to\infty} \|f_n\|_{\cN_0}$. Here, $\cN$ denote either $\BK$ or $\BK^\#$.

Since the rest of the completion process is the same as in \cite{Songl^1},
we only have a quick review and conclude the followings without proof.

\ \\
By \textbf{Proposition \ref{Proppointevaluation}} and \cite[Propositions 2.3 and 3.1]{Songl^1}, we have
$$
   ({\bf A\ref{A3}})~~\Longleftrightarrow~~\|\cdot\|_{\BK}~\text{is~well~defined}~\mbox{and}~\BK~\text{is~a~pre-RKBS}.
$$
By \textbf{Proposition \ref{PropnormBK*notcomplete}} and \cite[Proposition 2.5 and Lemma 3.3]{Songl^1}, we have
$$
 \|\cdot\|_{\BK^\#}~\text{is~well~defined}~\mbox{and}~ \BK^\#~\text{is~a~pre-RKBS}.
$$
Moreover, the bilinear form could be extended uniquely to $\BK\times\BK^\#$
such that the reproducing property in {Definition \ref{defiRKBS}} holds true.
That is,
\begin{equation}\label{bilinearfg}
  \left(f,{\K}(\cdot,x)c\right)_{{\K}}=\langle  f(x),c\rangle_q~\text{and}~\left({\K}(x,\cdot)c,g\right)_{{\K}}=\langle  c,g(x)\rangle_p
\end{equation}
for every $x\in X,c\in \cB_p,f\in\BK,g\in\BK^\#$.

We conclude the above discussion as follows.
\begin{Thm}\label{theoremwithoutproof}
Let ${\K}:X\times X\to \cL(\cB_p,\cB_q)$ be a kernel function satisfying ({\bf A\ref{A2}}) and ({\bf A\ref{A3}}).
Then the spaces $\BK$ and $\BK^\#$ defined in (\ref{BK}) and (\ref{BK*}), respectively,  satisfy
\begin{enumerate}
\renewcommand{\labelenumi}{(\roman{enumi})}
  \item they are both RKBSs of vector-valued functions from $X$ to $\cB_q$ with ${\K}$ being their reproducing kernel;
  \item the bilinear form (\ref{bilinearfinite}) can be extended to $\BK\times\BK^\#$,
  which satisfies the reproducing property (\ref{bilinearfg}) and
          \begin{equation}\label{Cauchyinequality}
            |(f,g)_{{\K}}|\le \|f\|_{\BK}\|g\|_{\BK^\#},\ \ f\in\BK,g\in\BK^\#.
          \end{equation}
\end{enumerate}
\end{Thm}

\section{The Representer Theorem}

The linear representer theorem is very important in regularized learning schemes in machine learning.
It enables us to transform the optimization problem in an infinite-dimensional space into one in a finite-dimensional space.
That is, the solution of the optimization problem can be represented explicitly as the finite
linear combinations of the kernels centered on the training samples.
  The representer theorem for the regularized learning schemes on RKBSs and for the minimal norm interpolations
  are often related \cite{Argyriou,Micchelli1,Songl^1,WangXu}.

In this section, we shall use the assumptions ({\bf A\ref{A1}}),
({\bf A\ref{A2}}) and ({\bf A\ref{A4}}) in {Definition \ref{defiadmissible}}
to deduce a corresponding representer theorem for the constructed vector-valued RKBSs $\BK$ and $\BK^\#$.

Recall that a linear operator between norm vector spaces $F:\cN_1\to \cN_2$ is said to be
\textit{completely continuous} \cite{Conway} on $\cN_1$,
if for any sequence $\{z_k\}\subseteq \cN_1$ weakly
convergent to $z_0\in\cN_1$, $\{F(z_k)\}$ converges to $F(z_0)$ strongly.
Note that every linear compact operator is completely continuous.
For example, the projection $P$ from an infinite-dimensional Banach space
to its finite-dimensional subspace is completely continuous.

\begin{Def}[{\bf Acceptable Regularized Learning Schemes}]\label{defiacceptableRLS}
Let $\x=\{x_i:i\in\bN_m\}\subseteq X$ be a set of pairwise distinct sampling points.
For any $f\in\BK$, define $f(\x)=(f(x_i):i\in\bN_m)^T\in\cB_q^m$.
Consider a loss function $L:\cB_q^m\times \cB_q^m\to \bR_+$ satisfying $L(\y,\y)=0$ for any $\y\in \cB_q^m$.
Let $\lambda>0$ and $\phi:\bR_+\to \bR_+$ be a nondecreasing function. A regularized learning scheme
\begin{equation}\label{acceptableRLS}
 \inf_{f\in\BK}\left\{L(f(\x),\y)+\lambda\phi(\|f\|_{\BK})\right\}
\end{equation}
is said to be \textbf{acceptable} in $\BK$ if $L$ is completely continuous on $\cB_q^m\times \cB_q^m$,
$\phi$ is continuous, and $\lim\limits_{t\to\infty}\phi(t)=+\infty$.
\end{Def}

Note that if the space $\cB_q$ is a finite-dimensional vector space or is the classical $\ell^1$ space, then strong continuity is equivalent to
continuity.

\begin{Def}
The space $\BK$ is said to satisfy the
\textbf{linear representer theorem for acceptable regularized learning}
if every acceptable regularized learning scheme (\ref{acceptableRLS}) has a minimizer of the form
\begin{equation}\label{LinearRT}
  f_0(x)=\sum_{j=1}^m {\K}(x_j,x)c_j,~~~\text{where}~x\in X,~c_j\in\cB_p,~j\in\bN_m.
\end{equation}
\end{Def}

 Denote
 $$
 \cS^\x:=\left\{\sum_{j=1}^m {\K}(x_j,\cdot)c_j:c_j\in\cB_p,j\in\bN_m\right\}.
 $$
 One should be aware that although the space $\cS^\x$ defined here is the ``span'' of $\{{\K}(x_i,\cdot):i\in\bN_m\}$
 with their coefficient in $\cB_p$, it may not be a finite-dimensional subspace of $\BK$.
  That is why we impose the complete continuity on $L$.

 A minimal norm interpolant in $\BK$ with respect to $(\x,\y)=\{(x_i,y_i):i\in\bN_m\}$ is a function $f_{\min}$ that is a minimizer of
 \begin{equation}\label{minimalNI}
   \inf\{\|f\|_{\BK}:f\in\cI_\x(\y)\},
 \end{equation}
  where $\cI_\x(\y):=\{f\in\BK:f(\x)=\y\}$. Unless stated otherwise, we assume that $f_{\min}$ always exists.

 \begin{Def}
 The space $\BK$ is said to satisfy the
 \textbf{linear representer theorem for minimal norm interpolation}
 if for arbitrary choice of training data $(\x,\y)$, the minimal norm interpolant
 $f_{\min}$ of (\ref{minimalNI}) lies in $\cS^\x.$
\end{Def}

 We will prove that the space $\BK$ satisfies the linear representer theorem for acceptable
 regularized learning if and only if it satisfies the linear representer
 theorem for minimal norm interpolation.

Let $\{e^i_p:i\in\bN\}$ be the canonical unit vectors of $\cB_p$.
That is, $e^i_p=(0,\dots,0,1,0,\dots)$ with the $i$-th coordinate $1$ and $0$ otherwise.
For $p\neq 2$, the set $\{e^i_p:i\in\bN\}$ is of finite cardinality.
Also, for $x=\sum\limits_{i} x_ie^i_p\in \cB_p,~y=\sum\limits_{i} y_ie^i_q\in \cB_q$, $\langle x,y\rangle_p=\sum\limits_{i} x_i y_i$.
With $m\in\bN$ being fixed, we denote $\be^i_p$ be the element $(e^i_p,e^i_p,\dots,e^i_p)\in\cB_p^m$.

\begin{Thm}\label{theoremARLequivalenttoMNI}
The space $\BK$ satisfies the linear representer theorem for acceptable regularized learning
if and only if it satisfies the linear representer theorem for minimal norm interpolation.
\end{Thm}
\begin{proof}
  \textbf{To prove the sufficiency}, let $\x=\{x_1,\dots,x_m\}\subseteq X$
  and $\y=(y_1,\dots,y_m)^T\in \cB_q^m$. Also
  let $L,\phi$ and $\lambda$ be as in the {Definition \ref{defiacceptableRLS}} that make (\ref{acceptableRLS}) an acceptable regularized
  learning scheme. For any $F\in\BK$, let $F_0$ be the minimizer of $\inf\limits_{g\in\cI_\x(F(\x))}\|g\|_{\BK}$
  with the form (\ref{LinearRT}). Then $F_0(\x)=F(\x)$ and $\|F_0\|_{\BK}\le \|F\|_{\BK}$.
  As a consequence, by the monotonicity of $\phi$,
 $$
 L(F_0(\x),\y)=L(F(\x),\y))~~\text{and}~~\phi(\|F_0\|_{\BK})\le \phi(\|F\|_{\BK}).
 $$
 It follows that
 $$
 \inf_{f\in\BK}\left\{L(f(\x),\y)+\lambda\phi(\|f\|_{\BK})\right\}=\inf_{f\in\cS^\x}\left\{L(f(\x),\y)+\lambda\phi(\|f\|_{\BK})\right\}.
 $$
 As $\lim\limits_{t\to+\infty} \phi(t)=+\infty$, there exists a large enough positive number $R$ such that
 $$
 \inf_{f\in\cS^\x}\left\{L(f(\x),\y)+\lambda\phi(\|f\|_{\BK})\right\}
 =\inf_{f\in\cS^\x,\atop\|f\|_{\BK}\le R}\left\{L(f(\x),\y)+\lambda\phi(\|f\|_{\BK})\right\}=:\eta.
 $$
  Set $\cF_R:=\{f:f\in\cS^\x~\text{and}~\|f\|_{\BK}\le R\}$. For every $k\in\bN$, there exists
 $f_k=\sum\limits_{i=1}^m {\K}(x_i,\cdot)c_i^k\in \cF_R$ such that
 $$
 L(f_k(\x),\y)+\lambda\phi(\|f_k\|_{\BK})\le \eta+\frac{1}{k}.
 $$
 Since $\|c_i^k\|_p\le R$ for any $i\in\bN_m,k\in\bN$ and $\{c:\|c\|_p\le R\}\subseteq \cB_p$
 is weakly compact, for each $i\in\bN_m$, there is a subsequence
  $\{c_i^{k_n}:n=1,2,\dots\}$ that converges  weakly to some $c_i^0\in\cB_p$.
 Thus, $\|c_i^0\|_p\le\liminf\limits_{n\to\infty} \|{c_i^{k_n}}\|_p$ for every $i\in\bN_m$.
 Therefore, by the assumption ({\bf A\ref{A2}}), for every $j\in\bN_m$,
  $f_0(x_j):=\sum\limits_{i=1}^m {\K}(x_i,x_j)c_i^0$ a weak limit of certain subsequence of $\{f_k(x_j):k=1,2,\dots\}$.

 We conclude that there is a subsequence
 $
   \{f_{k_n}(\x):n=1,2,\dots\}$ that converges weakly to $f_0(\x)\in\cB_q^m~$ and
 $$
   \|f_0\|_{\BK}\le \liminf\limits_{n\to\infty} \|f_{k_n}\|_{\BK}.
 $$
 Note that the functional in the brace of (\ref{acceptableRLS}) is continuous with
 respect to $f\in\BK$ by the assumption on $L,\phi$ and the continuity of the point evaluation functionals on $\BK$. Therefore, $f_0\in\cS^\x$ is a minimizer of (\ref{acceptableRLS}), which implies that $\BK$ satisfies the linear representer theorem for acceptable regularized learning.

 \textbf{Turning to the necessity,} for any $m$ pairwise distinct samplings $\x=\{x_1,\dots,x_m\}\subseteq X$
 and observations $\y=(y_1,\dots,y_m)^T\in \cB_q^m$,
 we choose $\{L_k:k=1,2,\dots\}$ and $\phi$ as
 \begin{equation}\label{L_k}
   L_n(f(\x),\y)=\|P_n(f(\x)-\y)\|_{q,1},\ \ \phi(t)=t,
 \end{equation}
   where $P_k$ is the canonical projection of $\cB_q^m$ which maps its element into the finite-dimensional
   subspace ${\rm span}\{\be^1_q,\be^2_q,\dots,\be^n_q\}$ and thus must be compact.
   Then $L_n$ is completely continuous for each $n$.

 Let $\lambda_n=\frac{1}{n}$ and $f_n\in\cS^\x$ be the minimizer of
 (\ref{acceptableRLS}) with the specified $L_n,\phi,\lambda_n$ in (\ref{L_k}).
 Then $f_n={\K}^\x(\cdot){\bc}_n$ for some ${\bc}_n\in\cB_q^m$.
 And therefore, for any interpolant $g\in\cI_\x(\y)$ and any $n\in\bN$,
 \begin{equation}\label{f_ng}
   \|P_n(f_n(\x)-\y)\|_{q,1}+\lambda_n\|f_n\|_{\BK}\le \|P_n(g(\x)-\y)\|_{q,1}+\lambda_n\|g\|_{\BK}=\frac{\|g\|_{\BK}}{n}.
 \end{equation}
 Since $g(\x)=\y$, we have by (\ref{f_ng}) that
 \begin{equation}\label{f_n le g}
   \|f_n\|_{\BK}\le \|g\|_{\BK}.
 \end{equation}
 If we let
 $$
  f_0={\K}^\x(\cdot){\K}[\x]^{-1}\y,
 $$
  then $f_0\in\cI_\x(\y)\cap \cS^\x$. Also, let $Q_i,~i\in\bN_m$ be the $i$-th coordinate projection
  from $\cB_q^m$ to $\cB_q$. For any $h:=\sum\limits_{l} h_l e^l_p\in\cB_p$,  we have by (\ref{f_ng}) for large enough $n\in\bN$ and every $i\in\bN_m$,
 \begin{equation}\label{f_nweaklyconvergestoy}
  \begin{split}
        &\quad~ |\langle  Q_i(f_n(\x)-\y),h \rangle_q|\\
        & \le |\langle  Q_i(f_n(\x)-P_nf_n(\x)),h \rangle_q|
           +|\langle  Q_iP_n(f_n(\x)-\y),h \rangle_q|+|\langle  Q_i(P_n\y-\y),h \rangle_q|\\
        & \le
      \begin{cases}
        \|f_n(\x)\|_{q,1}\left(\sum\limits_{l=n+1}^\infty |h_l|^p\right)^{\frac{1}{p}}
            +\frac{\|g\|_{\BK}}{n} \|h \|_p+\|\y\|_{q,1}\left(\sum\limits_{l=n+1}^\infty |h_l|^p\right)^{\frac{1}{p}}, & p=q=2\\
        \frac{\|g\|_{\BK}}{n} \|h \|_p, &p\neq q
      \end{cases}.
  \end{split}
 \end{equation}
  Then equation (\ref{f_nweaklyconvergestoy}) implies that $\{f_n(\x):n=1,2,\dots\}$ converges weakly to $\y$ in $\cB_q^m$.
 And therefore,  $\{{\K}[\x]^{-1}f_n(\x):n=1,2,\dots\}$ converges weakly to ${\K}[\x]^{-1}\y$ on $\cB_p^m$.
 Thus,
 \begin{equation*}
   \|{\K}[\x]^{-1}\y\|_{p,1}\le \liminf\limits_{n\to\infty}\|{\K}[\x]^{-1}f_n(\x)\|_{p,1}.
 \end{equation*}
 As a consequence,
   \begin{equation}\label{f_nc_n}
 \begin{split}
     \|f_0\|_{\BK}
   & =\|{\K}[\x]^{-1}\y\|_{p,1}\le \liminf_{n\to\infty}\|{\K}[\x]^{-1}f_n(\x)\|_{p,1}\\
   & = \liminf_{n\to\infty}\|\bc_n\|_{p,1}\le \limsup_{n\to\infty}\|f_n\|_{\BK}.
   \end{split}
 \end{equation}
 Combining (\ref{f_n le g}) and (\ref{f_nc_n}), we obtain $\|f_0\|_{\BK}\le \|g\|_{\BK}$.
 That is, $f_0$ is a solution of (\ref{minimalNI}). The proof is complete.
\end{proof}

Therefore, examining the relationship between assumption ({\bf A\ref{A4}}) and acceptable regularized learning scheme is equivalent to exploring the connection between ({\bf A\ref{A4}}) and minimal norm interpolation problem. The benefit of establishing this equivalence is that the minimal norm interpolation problem appears to be much simpler to handle. The following lemma confirms this observation.

\begin{Lem}\label{lemSxSx^-}
Let $\x=\{x_1,x_2,\dots,x_m\}$ consist of pairwise distinct elements in $X$,
let $x_{m+1}\in X\sm \x$, and set $\overline \x=\x\cup \{x_{m+1}\}$.
The minimal norm interpolant in $\cS^\x$ is the same as the minimal norm interpolant in $\cS^{\overline{\x}}$, that is,
\begin{equation}\label{S^xS^x^-}
  \min_{f\in\cI_\x(\y)\cap\cS^{\x}}\|f\|_{\BK}
  =\min_{f\in\cI_\x(\y)\cap\cS^{\overline\x}}\|f\|_{\BK}~~\text{for~every}~ \y\subset \cB_q^m,
\end{equation}
if and only if the kernel ${\K}$ satisfies (\bf{A\ref{A4}}).
\end{Lem}
\begin{proof}
Every $f\in \cI_\x(\y)\cap\cS^\x$ has the form
$$
f(x)={\K}^\x(x)\cdot{\bc}=\sum\limits_{i=1}^m{\K}(x_i,x)c_i,~x\in X
$$
and
$$
 y_k=f(x_k)=\sum\limits_{i=1}^m{\K}(x_i,x_k)c_i,~\text{for~all~}k\in\bN_m.
$$
Combining the above two equations, one obtains $\y={\K}[\x]{\bc}$.
Similarly, let $g\in \cI_\x(\y)\cap\cS^{\overline\x}$ and $b:=g(x_{m+1})$,
then $g(x)={\K}^\x(x)\cdot{\bd}=\sum\limits_{i=1}^{m+1}{\K}(x_i,x)d_i,~x\in X$ and
$\left({\by\atop b}\right)={\K}[\overline\x]{\bd}$.

Since ${\K}[\x]$ and ${\K}[\overline\x]$ are both invertible, the operator
$$
\bs:={\K}(x_{m+1},x_{m+1})-{\K}^\x(x_{m+1}){\K}[\x]^{-1}{\K}_\x(x_{m+1})\in\cL(\cB_p,\cB_q)
$$
is invertible.
Furthermore, if we let $\overline\y=\left({\by\atop b}\right)\in\cB_q^{m+1}$, then by (\ref{matrixinversion}),
\begin{equation}\label{K(x)y+1}
  \begin{split}
  &\ \quad{\K}[\overline\x]^{-1}\overline\y\\
  &=\left[
      \begin{array}{cc}
        {\K}[\x] & {\K}_\x(x_{m+1}) \\
        {\K}^\x(x_{m+1}) & {\K}(x_{m+1},x_{m+1}) \\
      \end{array}
    \right]^{-1} \left(
                   \begin{array}{c}
                     \y \\
                     b \\
                   \end{array}
                 \right)\\
  &=\left[
      \begin{array}{cc}
        {\K}[\x]^{-1}+{\K}[\x]^{-1}{\K}_\x(x_{m+1})\bs^{-1}{\K}^\x(x_{m+1}){\K}[\x]^{-1} & -{\K}[\x]^{-1}{\K}_\x(x_{m+1})\bs^{-1} \\
        -\bs^{-1}{\K}^\x(x_{m+1}){\K}[\x]^{-1} & \bs^{-1} \\
      \end{array}
    \right] \left(
                   \begin{array}{c}
                     \y \\
                     b \\
                   \end{array}
            \right)\\
  &=\left(
           \begin{array}{c}
           {\K}[\x]^{-1}\y+{\K}[\x]^{-1}{\K}_\x(x_{m+1})\bs^{-1}\bt \\
            -\bs^{-1}\bt \\
           \end{array}
   \right),
  \end{split}
\end{equation}
where $\bt:={\K}^\x(x_{m+1}){\K}[\x]^{-1}\y-b\in\cB_q.$

We are now in a position to show the sufficiency.
If the assumption ({\bf A\ref{A4}}) holds true, then by (\ref{K(x)y+1}),
\begin{equation*}
\begin{split}
  \|g\|_{\BK}
&= \|{\K}[\overline\x]^{-1}\overline\y\|_{p,1}\\
&\ge \|{\K}[\x]^{-1}\y\|_{p,1}-\|{\K}[\x]^{-1}{\K}_\x(x_{m+1})\bs^{-1}\bt\|_{p,1}+\|\bs^{-1}\bt\|_{p}\\
&\ge \|{\K}[\x]^{-1}\y\|_{p,1}-\|{\K}[\x]^{-1}{\K}_\x(x_{m+1})\|_{p,1}\cdot\|\bs^{-1}\bt\|_{p}+\|\bs^{-1}\bt\|_{p}\\
&\ge \|{\K}[\x]^{-1}\y\|_{p,1}.
\end{split}
\end{equation*}
That is, for every $x_{m+1}\in X\sm \x$,
$$
 \min_{f\in\cI_\x(\y)\cap\cS^{\overline\x}} \|f\|_{\BK}
 \ge \min_{f\in\cI_\x(\y)\cap\cS^\x} \|f\|_{\BK}~~\text{for~all}~\y\in\cB_q^m.
$$
Because $\cS^\x\subseteq \cS^{\overline\x}$, the reverse inequality is straightforward. Therefore, (\ref{S^xS^x^-}) holds.

To prove the necessity, we know by (\ref{S^xS^x^-}) that for any $\y\in\cB_q^m,b\in\cB_q$,
\begin{equation}\label{KxKx^-}
  \|{\K}[\overline\x]^{-1}\overline\y\|_{p,1}\ge \|{\K}[\x]^{-1}\y\|_{p,1}.
\end{equation}
For any $c\in \cB_p$ with $\|c\|_p=1$, if we choose
$$
\y={\K}_\x(x_{m+1})c\in\cB_q^m,~b=[{\K}^\x(x_{m+1}){\K}[\x]^{-1}{\K}_\x(x_{m+1})+\bs]c\in\cB_q.
$$
Then by (\ref{K(x)y+1}),
\begin{equation}\label{e_j}
   \|{\K}[\overline\x]^{-1}\overline\y\|_{p,1}
 =\left\|\begin{array}{c}
    \0 \\
     c
   \end{array}
 \right\|_{p,1}=1.
\end{equation}
Therefore, by (\ref{KxKx^-}) and (\ref{e_j}),
\begin{equation*}
   \begin{split}
     1  &=\sup_{\|c\|_p=1}\|{\K}[\overline\x]^{-1}\overline\y\|_{p,1}
            \ge\sup_{\|c\|_p=1}\|{\K}[\x]^{-1}\y\|_{p,1}\\
        &=\sup_{\|c\|_p=1}\|{\K}[\x]^{-1}{\K}_\x(x_{m+1})c\|_{p,1}\\
        &=\|{\K}[\x]^{-1}{\K}_\x(x_{m+1})\|_{p,1}.
   \end{split}
\end{equation*}
This yields ({\bf A\ref{A4}}) and thus completes the proof.
\end{proof}

\begin{Thm}\label{thmMNI}
Every minimal norm interpolant of (\ref{minimalNI}) in $\BK$
satisfies the linear representer theorem if and only if ({\bf A\ref{A4}}) holds true.
\end{Thm}
\begin{proof} We start with the necessity part. Note that the minimal norm interpolant in (\ref{minimalNI}) satisfying the linear representer theorem if and only if
\begin{equation}\label{g=f}
   \min_{g\in\cI_\x(\y)} \|g\|_{\BK} = \min_{f\in\cI_\x(\y) \cap \cS^\x} \|f\|_{\BK}.
\end{equation}
Hence, if the above equality holds, then since
\(\cI_\x(\y) \cap \cS^\x \subseteq \cI_\x(\y) \cap \cS^{\overline{\x}} \subseteq \cI_\x(\y)\),
we derive (\ref{S^xS^x^-}), and by applying Lemma \ref{lemSxSx^-},
the assumption ({\bf A\ref{A4}}) is satisfied for every \(x_{m+1} \in X \setminus \x\).

 Turning to the sufficiency, we notice
 $$
  \min_{g\in\cI_\x(\y)} \|g\|_{\BK}\le \min_{f\in\cI_\x(\y)\cap \cS^\x} \|f\|_{\BK}.
 $$
 We have to show that the reverse of the above inequality. To this end, for any $g\in\cI_{\x}(\y)\cap \cB_0$, we can express $g$ as
 $g=\sum\limits_{i=1}^n {\K}(x_i,\cdot)c_i$ for some $n\ge m$ and pairwise distinct $x_i\in X,~c_i\in\cB_p,~i\in\bN_n$.
 This is true since we can always add extra samplings from $X\sm\x$ by setting the corresponding coefficients
 $c_i$ to be zero, and relabeling if necessary. Let $y_j=g(x_j),m+1\le j\le n$ and
 $$
  \x_l:=(x_i:i\in\bN_l)^T,~~\y_l:=(y_i:i\in\bN_l)^T~~\text{for~every}~m\le l\le n.
 $$
 Note that $\x=\x_m$ and $\y=\y_m$ and $g\in\cI_{\x_n}(\y_n)\cap \cS^{\x_n}$. Therefore,
 $$
  \|g\|_{\BK}\ge \min_{f\in \cI_{\x_n}(\y_n)\cap \cS^{\x_n}} \|f\|_{\BK}.
 $$
 Also, by Lemma \ref{lemSxSx^-} and $\cI_{\x_n}(\y_n)\subseteq \cI_{\x_{n-1}}(\y_{n-1})$,
 $$
  \min_{f\in\cI_{\x_n}(\y_n)\cap \cS^{\x_n}}\|f\|_{\BK}
  \ge \min_{f\in\cI_{\x_{n-1}}(\y_{n-1})\cap \cS^{\x_n}} \|f\|_{\BK}
  =\min_{f\in\cI_{\x_{n-1}}(\y_{n-1})\cap \cS^{\x_{n-1}}} \|f\|_{\BK}.
 $$
 Thus, we have
 $$
  \|g\|_{\BK}\ge \min_{f\in \cI_{\x_{n-1}}(\y_{n-1})\cap \cS^{\x_{n-1}}} \|f\|_{\BK}.
 $$
 One can repeat this process until (\ref{g=f}) holds true for $g\in\cI_{\x}(\y)\cap \cB_0$.

 For a general $g\in\cI_{\x}(\y)$, let $\{g_k\in \cB_0:k\in\bN\}$ be the sequence that converges to $g$ in $\BK$.
 If we take $f,f_k\in\cS^\x$ as follows
 $$
  f(x)={\K}^\x(x){\K}[\x]^{-1}g(\x)~~\text{and}~~ f_k(x)={\K}^\x(x){\K}[\x]^{-1}g_k(\x),~~k\in\bN.
 $$
 Since $\|g_k-g\|_{\BK}\to 0$ as $k\to\infty$ and the point evaluation functionals are continuous on $\BK$,
 $g_k(x_i)\to g(x_i)$ for $i\in \bN_m$ as $k\to\infty$. As a consequence,
 \begin{equation*}
 \begin{split}
    \lim\limits_{k\to\infty}\|f-f_k\|_{\BK}
   &= \lim\limits_{k\to\infty}\left\|{\K}[\x]^{-1}(g(\x)-g_k(\x))\right\|_{p,1}=0.
 \end{split}
 \end{equation*}
  Since we already knew that $\|g_k\|_{\BK}\ge \|f_k\|_{\BK}$ for all $k\in\bN$, by taking the limit about $k$, it holds
  $$
  \|g\|_{\BK}\ge \|f\|_{\BK}~~~\text{for~any}~g\in\cI_\x(\y).
  $$
   The proof is complete.
\end{proof}

Combining Theorem \ref{theoremwithoutproof} with Theorems \ref{theoremARLequivalenttoMNI} and \ref{thmMNI},
we have the following corollary for any $p\in [1,+\infty]$.
\begin{corollary}\label{corA4}
Let ${\K}:X\times X\to\cL(\cB_p,\cB_q)$ satisfy
({\bf A\ref{A1}})-({\bf A\ref{A3}}) as in Definition \ref{defiadmissible}. Then it induces a RKBS $\BK$ and
the following three statements are equivalent:
\begin{enumerate}
\renewcommand{\labelenumi}{(\alph{enumi})}
  \item The kernel ${\K}$ satisfies the assumption ({\bf A\ref{A4}}).
  \item\label{ARL min} Every acceptable regularized learning scheme in $\BK$ of the
                          form (\ref{acceptableRLS}) has a minimizer of the form (\ref{LinearRT}).
  \item\label{MNI min} Every minimal norm interpolant (\ref{minimalNI}) in $\BK$ satisfies the linear representer theorem.
\end{enumerate}
\end{corollary}

We remark that similar arguments can prove that if $\K$ satisfies ({\bf A\ref{A4}}),
then $\BK^\#$ also satisfies the linear representer theorem for acceptable regularized learning. Let us conclude this section with the following theorem.
\begin{Thm}\label{thmadmissilbe}
If ${\K}$ is an admissible kernel on $X\times X$, then $\BK$ and $\BK^\#$ as defined in Section 3 are both vector-valued RKBSs on
$X$. And the bilinear form $(\cdot,\cdot)_{\K}$ satisfies the reproducing property (\ref{bilinearfg}) and the Cauchy inequality (\ref{Cauchyinequality}). Furthermore, every acceptable regularized learning scheme as in Definition \ref{defiadmissible} possesses a minimizer $f_0$ of the form
$$
f_0(x)=\sum_{i=1}^m {\K}(x_i,x)c_i,~~x\in X
$$
for some $c_i\in\cB_p,i\in\bN_m$.

The converse is also true, that is, for the constructed spaces $\BK$ and $\BK^\#$ to
enjoy the above properties, ${\K}$ must be an admissible kernel on $X\times X$.
\end{Thm}

\section{Admissible Kernels}
We have seen that admissible kernels are fundamental to our construction. We present examples of admissible kernels in this section.

Recall the term $\|{\K}[\x]^{-1}{\K}_\x(x)\|_{p,1}$ in ({\bf A\ref{A4}}),
which is usually referred to as the Lebesgue constant \cite{Lebesgueconstant} of the kernel ${\K}$. It measures the stability of the kernel-based interpolation \cite{DeMarchi}.

 Define
\begin{equation*}
  \Lambda^s({\G}):=\sup_{{\bf w}\subseteq X}\Lambda^s_{\bf w}({\G})
  =\sup_{{\bf w}\subseteq X}\sup_{x\in X}\|{\G}[{\bf w}]^{-1}{\G}_{\bf w}(x)\|_{s}
\end{equation*}
to be the Lebesgue constant of the kernel $\G:X\times X\to\bC$, where ${\bf w}$ is a finite subset of $X$ and $\|\cdot\|_s$ is some specified norm. We desire for kernels $\G$ such that
$$
\Lambda^{p,1}_\x({\G})\le 1~~\text{for~every~pairwise~distinct~}\x\subseteq X.
$$

It is shown in \cite{Songl^1} that both the Brownian bridge kernel
 $$
  K_{\min}(x,x')=\min\{x,x'\}-xx',~~x,x'\in (0,1)
 $$
 and the exponential kernel
  $$
   K_{\exp}(x,x')=\exp(-|x-x'|),~~x,x'\in \bR
  $$
    are admissible scalar-valued kernels.
    Here we present a new family of admissible scalar-valued kernels.
    They can then be utilized to construct admissible operator-valued kernels
       \begin{equation}\label{KA}
            {\G}(x,x')=G(x,x')\bA,
       \end{equation}
      for our purpose, where $G:X\times X\to\bC$ is an admissible scalar-valued kernel and $\bA$ denotes a positive definite matrix.

\subsection{New Admissible Kernels}
The first new family of scalar-valued admissible Kernels is
\begin{equation}\label{K_t}
  K_{t}(x,y)=\min\{x,y\}-t\,xy, ~~\text{where}~x,y\in (0,1),~-1\le t\le 1.
\end{equation}
It contains the Brownian bridge kernel $K_{\min}$ when $t=1$.
When $t=0$, it is the covariance of the Brownian motion.

\begin{Prop}\label{K_t admissible}
The family of functions $K_t$ in (\ref{K_t}) are admissible kernels.
\end{Prop}
\begin{proof}
Let $m\in \bN$ and $0<x_1<x_2<\cdots<x_m<1$. A straightforward calculation
shows that the determinant of the kernel matrix
$K_t[\x]$ is given by
$$
x_1(1-tx_m)(x_2-x_1)(x_3-x_2)\cdots(x_m-x_{m-1}).
$$
It follows that $K_t$ is strictly positive definite for any $-1\le t\le 1$ and thus it satisfies the assumption ({\bf A\ref{A1}$'$}).
 Moreover, the function $K_t$ is clearly uniformly bounded by $2$ for $t\in[-1,1]$.
 Additionally, using the same arguments as in \cite[Proposition 5.1]{Songl^1}, one can verify that \( K_t \) satisfies ({\bf A\ref{A3}$'$}) and ({\bf A\ref{A4}$'$}) for $t\in [-1,1]$.
\end{proof}

To present the second family of admissible kernels, we shall use the Wendland's kernel function \cite{Wendland}, which has some well-behaved properties and is widely used in interpolation and kernel-based learning problems.
 We consider the restriction form of the Wendland's function
\begin{equation}\label{K_W}
  K_{w}(x,y):=\max\{1-|x-y|,0\}=1-|x-y|, ~\text{where}~x,y\in (0,1).
\end{equation}

It is straightforward to show that $K_w$ satisfies ({\bf A\ref{A1}$'$}), ({\bf A\ref{A2}$'$}), and ({\bf A\ref{A3}$'$}). We only need to focus on its Lebesgue constants. We shall show that positive linear combinations of $K_t$ and $K_w$ have Lebesgue constants
bounded above by 1.

To advance our discussion, we first introduce some fundamental concepts from convex analysis \cite{convexanalysis}. Recall that a subset \( E \) of a linear space is termed \emph{convex} if, for any elements \( x, y \in E \) and any \( t \in [0,1] \), the convex combination \( t x + (1 - t) y \) also belongs to \( E \). The \emph{convex hull} of the set \( E \) is defined by
\[
\mathrm{conv}\, E := \{ t x + (1 - t) y : x, y \in E, \, t \in [0,1] \}.
\]

A function \( f : X \to \mathbb{R} \), where \( X \) is a linear space, is called \emph{convex} if, for every real number \( \alpha \), the sublevel set \( \{ x \in X : f(x) \leq \alpha \} \) is convex. It is noteworthy that linear functions map convex sets to convex sets.

Assuming that the operator \( K[\mathbf{x}]^{-1} \) is a bounded linear transformation, the verification of condition \textbf{(A4\('\))} for a kernel \( K \) reduces to establishing the following two properties:
\begin{enumerate}
\renewcommand{\labelenumi}{(\textbf{K\arabic{enumi}}')}
\item For each index \( i \in \mathbb{N}_m \),
\[
\| K[\mathbf{x}]^{-1} K_{\mathbf{x}}(x_i) \|_1 \leq 1;
\]
\item For every \( x \in X \), the element \( K_{\mathbf{x}}(x) \) lies within the convex hull of \( \{ \pm K_{\mathbf{x}}(x_i) : i \in \mathbb{N}_m \} \).
\end{enumerate}

Finally, recall that hyperplanes in \( \mathbb{R}^k \) can be characterized as the zero sets of linear equations in the variables \( u_1, u_2, \ldots, u_k \), that is,
\[
a_1 u_1 + a_2 u_2 + \cdots + a_k u_k + a_0 = 0, \quad a_i \in \mathbb{R}, \quad i = 0, 1, \ldots, k.
\]

\begin{Thm}\label{propKWA4}
The Wendland kernel $K_w$ in (\ref{K_W}) satisfies ({\bf A\ref{A4}$'$}).
\end{Thm}
\begin{proof}
Let $m\in N$ and $0<x_1<x_2<\cdots<x_m<1$.
Since $K_w$ is positive definite \cite{Wendland},
the kernel matrix
$$
 K_w[\x]=\left[
             \begin{array}{ccccc}
                1           &   1-x_2+x_1     & 1-x_3+x_1      & \cdots  & 1-x_m+x_1     \\
                1-x_2+x_1   &   1             & 1-x_3+x_2      & \cdots  & 1-x_m+x_2     \\
                1-x_3+x_1   &   1-x_3+x_2     & 1              & \cdots  & 1-x_m+x_3     \\
                 \vdots     &  \vdots         & \vdots         & \ddots  & \vdots        \\
                1-x_m+x_1   &   1-x_m+x_2     & 1-x_m+x_3      & \cdots  & 1             \\
             \end{array}
         \right]
$$
is strictly positive definite. Therefore the points
$\{(K_w)_{\x}(x_i):i\in\bN_m\}\subseteq\bR^m$ determine an $(m-1)$-dimensional hyperplane $\pi_0$.
For simplicity, we use $K$, $K_{i,j}$ and $K_{x,j}$ to denote $K_w$, $K_w(x_i,x_j)$ and $K_w(x,x_j)$, respectively.
 Then it is readily to see that any point $(u_1,u_2,\dots,u_m)$ in $\pi_0$ must satisfy
\begin{equation*}\label{pi}
  \Pi(u_1,u_2,\dots,u_m)
: =\det\left[
             \begin{array}{ccccc}
                1   &  u_1        & u_2        & \cdots & u_m     \\
                1   &  K_{1,1}  & K_{1,2}  & \cdots & K_{1,m} \\
                1   &  K_{2,1}  & K_{2,2}  & \cdots & K_{2,m} \\
             \vdots &  \vdots     & \vdots     & \ddots & \vdots  \\
                1   &  K_{m,1}  & K_{m,2}  & \cdots & K_{m,m} \\
             \end{array}
         \right] =0.
\end{equation*}
Let
$$
 F (x):=\Pi (K_\x (x))=
 \det\left[
             \begin{array}{ccccc}
                1    &  K_{x,1}  & K_{x,2}  & \cdots & K_{x,m} \\
                1    &  K_{1,1}  & K_{1,2}  & \cdots & K_{1,m} \\
                1    &  K_{2,1}  & K_{2,2}  & \cdots & K_{2,m} \\
              \vdots &    \vdots   & \vdots     & \ddots & \vdots  \\
                1    &  K_{m,1}  & K_{m,2}  & \cdots & K_{m,m} \\
             \end{array}
 \right].
$$
For any $E\subseteq [0,1]$, denote
$$
\Gamma_{E}:=\{K_\x (x):x\in E\}\subseteq \bR^m.
$$

Obviously, $\|K[\x]^{-1}K_{\x}(x_i)\|_{1}=1$ for every $i\in\bN_m$. Now let $x\in (0,1)$ be different from $x_j,j\in\bN_m$. We first consider the case when $x\in (x_j,x_{j+1})$, $j\in \bN_{m-1}$.
A direct calculation shows that for $x_j<x<x_{j+1}$,
\begin{equation}\label{A4middle}
  \|K[\x]^{-1}K_{\x}(x)\|_{1}=\left\|(0,\dots,0,\frac{x_{j+1}-x}{x_{j+1}-x_j},\frac{x-x_j}{x_{j+1}-x_j},0,\dots,0)^T\right\|_{1}=1
\end{equation}
Next consider cases when $0< x<x_1$ and $x_m< x < 1$.
 It suffices to prove ({\bf K2$'$}) in these two cases. To this end, we will prove that $\0$, $\Gamma_{(0,x_1)}$ and $\Gamma_{(x_m,1)}$ are on the same side of the
hyperplanes which pass through
\begin{equation}\label{b_i}
  \left\{b_iK_{\x}(x_i):~b_i\in\{\pm 1\}, i\in\bN_m\right\}.
\end{equation}
 Since the original point $\0$ lies in the convex cone of $\{\pm K_\x(x_i):i\in\bN_m\}$, so does
 $\Gamma_{(0,x_1)}$ and $\Gamma_{(x_m,1)}$.

We only prove the result for the hyperplane $\pi_0$ corresponding to $b_i=1$, $i\in\bN_m$ in (\ref{b_i}) as other hyperplanes can be handled in a similar manner.
Note that
$$
\Pi(\0)=\det\left[
             \begin{array}{ccccc}
                1    &  0        & 0        & \cdots & 0       \\
                1    &  K_{1,1}  & K_{1,2}  & \cdots & K_{1,m} \\
                1    &  K_{2,1}  & K_{2,2}  & \cdots & K_{2,m} \\
             \vdots  &  \vdots   & \vdots   & \ddots & \vdots  \\
                1    &  K_{m,1}  & K_{m,2}  & \cdots & K_{m,m} \\
             \end{array}
            \right]>0.
$$
Also for $x\in (0,x_1)$,
$$
{\rm d} F(x)/{\rm d} x=\det\left[
             \begin{array}{ccccc}
                0    &  1        & 1        & \cdots & 1       \\
                1    &  K_{1,1}  & K_{1,2}  & \cdots & K_{1,m} \\
                1    &  K_{2,1}  & K_{2,2}  & \cdots & K_{2,m} \\
             \vdots  &  \vdots   & \vdots   & \ddots & \vdots  \\
                1    &  K_{m,1}  & K_{m,2}  & \cdots & K_{m,m} \\
             \end{array}
            \right]<0.
$$
Similarly, ${\rm d} F(x)/{\rm d} x>0$ for $x\in (x_m,1)$. Combining this with the facts
$$
F(x_1)=F(x_m)=0,
$$
we have that $F(x)\ge 0$ for $(0,x_1)\cup (x_m,1)$.
That is, $\Gamma_{(0,x_1)}$, $ \Gamma_{(x_m,1)}$, and ${\0}$ locate on the same side of $\pi_0$.

We conclude that ({\bf K1$'$}) and ({\bf K2$'$}) are satisfied, and thus ({\bf A\ref{A4}$'$}) holds true.
\end{proof}

Positive linear combinations of $K_t$ and $K_{w}$ are also admissible.
\begin{Prop}\label{propC1K+C2Kadmissible}
The following class of kernels
\begin{equation}\label{C1K+C2Kadmissible}
  K:=C_1 K_t + C_2 K_w,\ \ C_1,C_2>0.
\end{equation}
satisfies ({\bf A\ref{A4}}$'$).
\end{Prop}
\begin{proof} It suffices to prove
$$
\alpha K_{t}+\beta K_w,~~\alpha,\beta\in (0,1)~\text{and}~\alpha+\beta=1,
$$
all satisfy ({\bf A\ref{A4}$'$}). Let $0< x_1<x_2<\cdots<x_m< 1$. We reuse the notations $F,\pi_0$ and $\Pi$ in the proof of Theorem \ref{propKWA4},
with the kernel here being $K:=\alpha K_{t}+(1-\alpha)K_w$.
One sees that $F'(x)$ is a function independent of $x$. Therefore it is constant for $x\in(x_i,x_{i+1}),~i\in\bN_{m-1}$.
Since
$$
\Pi(K_\x(x_i))=F(x_i)=F(x_{i+1})=\Pi(K_\x(x_{i+1}))=0,
$$
$F'(x)=0$. That is, $\Gamma_{[x_1,x_m]}\subseteq \pi_0$.

The discussions for the two cases $x\in(0,x_1)$ and $x\in(x_m,1)$ are similar. Therefore we just show the first case with $b_i=1$ in (\ref{b_i}) for every $i\in\bN_m$. Note that the function $F'(x)$ is also a constant on the interval $(0,x_1)$ and
$$
 \Pi(\0)=\det K[\x]>0.
$$
By similar reasoning as in Proposition \ref{propKWA4},
\begin{equation*}
  \begin{split}
  {\rm d} F(x)/{\rm d} x
      =\det\left[
             \begin{array}{ccccc}
                0   & \beta        &\beta          & \cdots & \beta       \\
                1   &  K_{1,1}     & K_{1,2}       & \cdots & K_{1,m}     \\
                1   &  K_{2,1}     & K_{2,2}       & \cdots & K_{2,m}     \\
             \vdots &  \vdots      & \vdots        & \ddots & \vdots      \\
                1   &  K_{m,1}     & K_{m,2}       & \cdots & K_{m,m}     \\
             \end{array}
            \right]<0.
  \end{split}
\end{equation*}
which means that $\Gamma_{(0,x_1)}$ and $\0$ are located on the same side of $\pi_0$. Likewise, one can show that $\Gamma_{(0,x_1)}\cup \Gamma_{(x_m,1)}$ and $\0$ locate on the same side of any hyperplane in (\ref{b_i}). The proof is complete.
\end{proof}

\subsection{Admissible kernel for multi-task learning}
We will show that the multi-task kernel defined in (\ref{KA}) is admissible whenever $G$ is.
Let $G$ be an scalar-valued kernel and $\bA\in\cL(\cB_p,\cB_q)$ be an invertible operator as in ({\bf A\ref{A2}}).
\begin{Lem}\label{lemLambdaxK}
Let ${\G}:X\times X\to\cL(\cB_p,\cB_q)$ be a multi-task kernel given as in (\ref{KA})
and $\x$ be a set of $m$ pairwise distinct points.
If the Lebesgue constant $\Lambda_\x^{p,1}(G)$ is bounded by $\alpha_m>0$, then so is $\Lambda_\x^{p,1}({\G})$.
\end{Lem}
\begin{proof} We compute that
\begin{equation*}
\begin{split}
  {\G}[\x]^{-1}{\G}_\x(x)
  &=[G(x_i,x_j)\bA:i,j\in\bN_m]^{-1}\cdot(G(x,x_i)\bA:i\in\bN_m)^T\\
  &=[{\rm diag}(\bA)[G(x_i,x_j)\bI_p:i,j\in\bN_m]]^{-1}{\rm diag}(\bA)(G(x,x_1)\bI_p,\dots,G(x,x_m)\bI_p)^T\\
  &=[G(x_i,x_j)\bI_p:i,j\in\bN_m]^{-1}{\rm diag}(\bA)^{-1}{\rm diag}(\bA)(G(x,x_1)\bI_p,\dots,G(x,x_m)\bI_p)^T\\
  &=[G(x_i,x_j)\bI_p:i,j\in\bN_m]^{-1}(G(x,x_1)\bI_p,\dots,G(x,x_m)\bI_p)^T\\
  \end{split}
\end{equation*}
where ${\rm diag}(\bA)$ is an $m\times m$ block diagonal matrix with $\bA$ being its diagonal entries.
Then
\begin{equation*}
  \begin{split}
     \Lambda^{p,1}_\x({\G})
 &= \max_{x\in X}\left\|(b_1(x)\bI_p,b_2(x)\bI_p,\dots,b_m(x)\bI_p)^T\right\|_{p,1}\\
 &\le \max_{x\in X}\sum_{i=1}^m |b_i(x)| =\Lambda_\x^{p,1}(G) \le \alpha_m,
   \end{split}
\end{equation*}
which completes the proof.
\end{proof}





 We are now ready to present the following proposition.
\begin{Thm}\label{propadmissibleKA}
Let $K$ be an admissible scalar-valued kernel,
and $\bA$ be an invertible operator in $\cL(\cB_p,\cB_q)$, then ${\K}=K\bA$ is an admissible matrix-valued kernel.
\end{Thm}
\begin{proof} Note that the assumption ({\bf A\ref{A1}}) follows from the fact that $\bA$ is invertible
 and $K$  is strictly positive. And ({\bf A\ref{A2}}) follows by noting
 $\|{\K}(x,x')\|_{\cL(\cB_p,\cB_q)}\le |K(x,x')|\cdot\|\bA\|_{\cL(\cB_p,\cB_q)}$. As for  ({\bf A\ref{A3}}), let $\{x_i:i\in \bN\}$ be pairwise distinct points in $X$ and $c_i\in\cB_p,~i\in \bN$.
Suppose that $\sum\limits_{i\in\bN} {\K}(x_i,x)c_i=\0$ for every $x\in X$.
Then
$$
 \sum\limits_{i\in\bN} {\K}(x_i,x)c_i=\bA\sum\limits_{i\in\bN} {K}(x_i,x)c_i=\0.
$$
As a consequence, $\sum\limits_{i\in\bN} {K}(x_i,x)(c_i)_k=0$ for every $k\in\bN$.
By the assumption on $K$, $(c_i)_k=0$ for every $i,k\in\bN$, which implies $c_i=0,i\in\bN$. Therefore, ({\bf A\ref{A3}}) is also satisfied. Finally, by Lemma \ref{lemLambdaxK} and ({\bf A\ref{A4}$'$}), ({\bf A\ref{A4}}) holds true.
\end{proof}

By the above theorem and the results in Section 5.1, for any invertible operator $\bA\in\cL(\cB_p,\cB_q)$,
$$
{\K}_t=K_{t}\bA,\ {\K}_{w}=K_{w}\bA,\ {\K}_{\exp}=K_{\exp}\bA
$$
and
$$
C_1{\K}_t+C_2{\K}_{w},\ \ C_1,C_2>0
$$
are all admissible multi-task kernels.

\section{A Relaxed Linear Representer Theorem}
As we have mentioned in Section 5, only a few kernels satisfy the admissible assumption ({\bf A\ref{A4}}).  This restriction excludes many widely used kernels.
This section aims at weaking the condition ``Lebesgue constants are unifromly bounded by 1''
to accommodate more kernels.

Let $f:X\to\cB_q$ and denote for a set $\z=\{(x_i,y_i):i=1,2\dots,m\}$
$$
 \cE_\z(f):=\frac{1}{m}\sum_{i=1}^m \|f(x_i)-y_i\|_{\cB_q}.
$$
Recall that $\BK$ satisfies the linear representer theorem provided that
$$
 \min_{f\in\cS^\x}\cE_\z(f)+\lambda\|f\|_{\BK}=\min_{f\in\BK}\cE_\z(f)+\lambda\|f\|_{\BK}.
$$
To weaken ({\bf A\ref{A4}}), we consider relaxing the above to
\begin{equation}\label{relaxedlinear}
  \min_{f\in\cS^\x}\cE_\z(f)+\lambda\|f\|_{\BK}\le\min_{f\in\BK}\cE_\z(f)+\lambda\beta_m\|f\|_{\BK}
\end{equation}
where $\beta_m$ is a constant depending on the sampling numbers and the kernel ${\K}$.

Then we have a corresponding relaxed version of Lemma \ref{lemSxSx^-}.
\begin{Lem}\label{lembeta_m}
If there exists some $\beta_m\ge 1$ such that for all $\y\in\cB_q^m$,
\begin{equation}\label{relaxedS^xS^x-}
 \min_{f\in\cI_\x(\y)}\|f\|_{\BK}
 \ge \frac{1}{\beta_m}\min_{f\in\cI_\x(\y)\cap\cS^\x}\|f\|_{\BK}.
\end{equation}
Then the relaxed linear representer theorem (\ref{relaxedlinear}) holds true for any completely continuous loss function
$L$ and regularization parameter $\lambda>0$.
\end{Lem}
\begin{proof} Suppose that (\ref{relaxedS^xS^x-}) is satisfied. Let $f_0$ be a minimizer of
$$
 \min_{f\in\BK} L(f(\x),\y)+\lambda\beta_m\|f\|_{\BK}.
$$
By (\ref{relaxedS^xS^x-}), we can choose $g\in\cS^\x$ such that $g(\x)=f_0(\x)$ and
$$
 \|g\|_{\BK}\le \beta_m \|f_0\|_{\BK}.
$$
Then
$$
 L(g(\x),\y)+\lambda\|g\|_{\BK}\le L(f_0(\x),\y)+\lambda\beta_m\|f_0\|_{\BK},
$$
and hence the proof is complete.
\end{proof}

Next, we shall give a characterization of (\ref{relaxedS^xS^x-}),
which gives rise to a weaker version of ({\bf A\ref{A4}}) for the relaxed representer theorem.

\begin{Thm}
For any $q\in [1,+\infty]$, (\ref{relaxedS^xS^x-}) holds true for any $\y\in\cB_q^m$ if and only if
\begin{equation}\label{relaxedbeta}
  \|{\K}[\x]^{-1}{\K}_\x(x)\|_{p,1}\le \beta_m~\text{for~all~}x\in X.
\end{equation}
\end{Thm}
\begin{proof} We employ analogous arguments to those presented in Lemma \ref{lemSxSx^-} and Theorem \ref{thmMNI}.

 Let $f_0={\K}^\x(\cdot){\K}[\x]^{-1}\y$ and $g$ be an arbitrary function in $\cI_\x(\y)\cap\cB_0$.
 As in Lemma \ref{lemSxSx^-}, we may assume
 $g\in \cI_\x(\y)\cap\cS^{\x\cup{\bt}}$ for some $\bt=\{t_i:i\in\bN_n\}\subseteq X\sm\x$.
 Let ${\bb}:=g({\bt})$, and
 $$
  \left({\K}_\x(\bt)\right)_{ij}=({\K}(t_j,x_i):i\in\bN_m,j\in\bN_n)~\text{and}~({\K}^\x(\bt))_{ji}=({\K}(x_i,t_j):i\in\bN_m,j\in\bN_n)
 $$
 be $m\times n$ and $n\times m$ operator matrices, respectively.
 Then
 \begin{equation}\label{g_BK}
    \|g\|_{\BK}
 =   \left\|\left[
                  \begin{array}{cc}
                  {\K}[\x]    &  {\K}_\x(\bt) \\
                  {\K}^\x(\bt)  & {\K}[\bt]   \\
                  \end{array}
           \right]^{-1} \left[\begin{array}{c}
                           \y \\
                           {\bb}
                         \end{array}
                        \right]
     \right\|_{p,1}
 =   \left\| \left[
                   \begin{array}{c}
                       {\K}[\x]^{-1}\y-{\K}[\x]^{-1}{\K}_\x (\bt)\widetilde{{\bb}} \\
                         \widetilde{{\bb}}
                   \end{array}
              \right]
     \right\|_{p,1},
  \end{equation}
  where
  $$
   \widetilde{{\bb}}=
   \left[{\K}[t]-{\K}^\x(\bt){\K}[\x]^{-1}{\K}_\x(\bt)\right]^{-1}[{\bb}-{\K}^\x(\bt){\K}[\x]^{-1}\y]\in\cB_p^n.
  $$

  For the necessity part, if (\ref{relaxedS^xS^x-}) holds true for all $\y\in\cB_q^m$, then we let $\y={\K}[\x,t]\widetilde{b}$
  by choosing $\bt$ to be a singleton $\{t\}$ and $\widetilde{b}\in\cB_p$ with $\|\widetilde{b}\|_p=1$ to obtain
  \begin{equation*}
     \begin{split}
        \|g\|_{\BK}=\left\|\left[\begin{array}{c}
           \0 \\
           \widetilde{b}
         \end{array}
   \right]\right\|_{p,1}
   &=1 \ge \frac{1}{\beta_m} \|f_0\|_{\BK}\\
   &=\sup_{\widetilde{b}}\frac{1}{\beta_m} \|{\K}[\x]^{-1}\y\|_{p,1}
          =\sup_{\widetilde{b}}\frac{1}{\beta_m} \|{\K}[\x]^{-1}{\K}_\x(t)\widetilde{b}\|_{p,1}\\
   &=\frac{1}{\beta_m} \|{\K}[\x]^{-1}{\K}_\x(t)\|_{p,1},
       \end{split}
  \end{equation*}
  which yields (\ref{relaxedbeta}).

  Conversely, suppose that (\ref{relaxedbeta}) is satisfied.
  Then we need to prove that for all $g\in\cI_\x(\y)$,
  $$
   \|g\|_{\BK}\ge \frac{1}{\beta_m} \|f_0\|_{\BK}=\frac{1}{\beta_m} \|{\K}[\x]^{-1}\y\|_{p,1}.
  $$
  Follow similar arguments as those in Theorem \ref{thmMNI}, we first let $g\in\cI_\x(\y)\cap \cB_0$ have the norm (\ref{g_BK}) for some $\bt$.
  If $\|{\K}[\x]^{-1}\y\|_{p,1}\le\beta_m \|\widetilde{{\bb}}\|_{p,1}$, then
  $$
   \|g\|_{\BK}\ge\|\widetilde{{\bb}}\|_{p,1}\ge \frac{1}{\beta_m} \|{\K}[\x]^{-1}\y\|_{p,1}.
  $$
   If $\|{\K}[\x]^{-1}\y\|_{p,1}>\beta_m \|\widetilde{{\bb}}\|_{p,1}$,
   then by (\ref{Matrixnormcompatible}) and (\ref{relaxedbeta}),
   \begin{equation*}
     \begin{split}
       \|g\|_{\BK}
   &\ge \|{\K}[\x]^{-1}\y\|_{p,1}-\|{\K}[\x]^{-1}{\K}_\x(t)\widetilde{{\bb}}\|_{p,1}+\|\widetilde{{\bb}}\|_{p,1}\\
   &\ge \|{\K}[\x]^{-1}\y\|_{p,1}-\max\limits_{j\in\bN_n}\left(\|{\K}[\x]^{-1}{\K}_\x(t_j)\|_{p,1}\right)\|\widetilde{{\bb}}\|_{p,1}
         +\|\widetilde{{\bb}}\|_{p,1}\\
   &\ge \|{\K}[\x]^{-1}\y\|_{p,1}-(\beta_m-1)\|\widetilde{{\bb}}\|_{p,1}\\
   &\ge \|{\K}[\x]^{-1}\y\|_{p,1}-(\beta_m-1)\frac{1}{\beta_m}\|{\K}[\x]^{-1}\y\|_{p,1}\\
   &=   \frac{1}{\beta_m}\|{\K}[\x]^{-1}\y\|_{p,1}.
     \end{split}
   \end{equation*}
  This shows that for $g\in\cI_\x(\y)\cap \cB_0$,
  $$
   \|g\|_{\BK}\ge \frac{1}{\beta_m} \|f_0\|_{\BK}=\frac{1}{\beta_m} \|{\K}[\x]^{-1}\y\|_{p,1}.
  $$
  Then by the same limiting process as in Theorem \ref{thmMNI}, we prove the sufficiency.
\end{proof}

The above result together with Lemma \ref{lembeta_m} provide a relaxation of the requirement ({\bf A\ref{A4}}).
 The relaxed requirement do accommodates more kernels.
 We give some examples.
 It was shown in \cite{DeMarchi} that
 if a translation invariant kernel $K(x,y)=\phi(x-y)$ on $\bR^d$ satisfies for some positive constants $c_1,c_2,M,\tau>0$ that
 $$
   c_1(1+\|\xi\|_2^{2})^{-\tau}\le \hat\phi(\xi)\le c_2(1+\|\xi\|_2^{2})^{-\tau},\ \ \|\xi\|_2>M,
 $$
 then its Lebesgue constants for quasi-uniform inputs are bounded by a multiple of $\sqrt{d}$. Scalar-valued kernels satisfying the above Fourier transform condition include
 the followings:
  \begin{enumerate}
    \item Poisson radial functions in \cite{Fornberg}
    $$
     \phi_d(r)=\frac{J_{d/2-1}(\varepsilon r)}{(\varepsilon r)^{d/2-1}},~d\in\bN,
    $$
    where $\varepsilon>0$ and $J_\nu(r)$ is the Bessel function of the first kind of order $\nu$.
    \item Mat\'{e}rn functions
$$
    \phi(r)=\frac{2^{1-\nu}}{\Gamma_\nu}r^{\nu}K_{\nu}(r),
$$
    where $K_\nu$ is~the~modified~Bessel~function.
    \item Wendland's compactly supported functions
     $$
     \phi(r)=(\max\{0,1-r\})^k p(r),\ \ k\ge 2,
     $$
     where $p$ is a special polynomial (see \cite{Wendland}).
  \end{enumerate}

One may use the above scalared-valued kernels to construct desired operator-valued
kernels admissible for relaxed linear representer theorems as in Section 5.


{\small
			
}
 \end{document}